\documentclass[12pt]{article}
\usepackage{amsmath}
\usepackage{amssymb}
\usepackage{amsthm}
\usepackage[english]{babel}
\usepackage{nccmath}
\newtheorem{theorem}{Theorem}
\newtheorem{lemma}{Lemma}
\newtheorem{remark}{Remark}

\begin{document}
\begin{center}
\section*{The Hardy-Littlewood theorem for double Fourier-Haar series from Lebesgue spaces $L_{\bar{p}}[0,1]$ with mixed metric and from net spaces $N_{\bar{p}, \bar{q}}(M)$}
\bf{A.N.~Bashirova, E.D.Nursultanov}
\end{center} 

\noindent {\bf Abstract.} In terms of the Fourier-Haar coefficients, a criterion is obtained for the function $f (x_1,x_2)$ to belong to the net space $N_{\bar{p},\bar{q}}(M)$ and to the Lebesgue space $L_{\bar{p}}[0,1]^2$ with mixed metric, where $1<\bar{p}<\infty$, $0<\bar{q}\leq\infty$, $\bar{p}=(p_1,p_2)$, $\bar{q}=(q_1,q_2)$, $M$ is the set of all rectangles in $\mathbb{R}^2$. We proved the Hardy-Littlewood theorem for multiple Fourier-Haar series.

\noindent {\bf Key words:} net space, Lebesgue space, anisotropic space, Fourier series, Haar system.

\noindent {\bf AMS Mathematics Subject Classification:} (42B05, 42B35, 46B70)

\section{Introduction}

In studying the relationship between the integrability of a function and the summability of its Fourier coefficients, the most striking example is the Parseval equality
$$
\int^1_0|f(x)|^2dx=\sum^\infty_{k=1}|a_k|^2,
$$
where $a_k$ are Fourier coefficients by trigonometric system.

In the case, when $f\in L_p$, $p\neq2$, here the Hardy-Littlewood inequalities hold: 
if $2\leq p<\infty$, then
$$
\left\|f\right\|_{L_p}^p\leq c_1\sum_{k=1}^{\infty}k^{p-2}|a_k|^p,
$$
if  $1\leq p\leq 2$, then
$$
c_2\sum_{k=1}^{\infty}k^{p-2}|a_k|^p\leq\left\|f\right\|_{L_p}^p.
$$

For a function $f$ from $L_p$, the lower bounds for $p>2$ and the upper bounds for $1<p\leq2$ are proved only under additional conditions.

Here we know the Hardy-Littlewood theorem \cite{Zigmund} for trigonometric series:

Let $1<p<\infty$ and $f\sim\sum_{k=0}^{\infty}a_k coskx$. If $\left\{a_k\right\}_{k=0}^\infty$ is  monotonically non-increasing sequence, or $f$ is monotone function, then for $f\in L_p[0,\pi]$ it is necessary and sufficient
$$
\sum_{k=0}^{\infty}k^{p-2}|a_k|^p<\infty,
$$
and the relation is fulfilled
$$
\left\|f\right\|_{L_p}^p\asymp\sum_{k=0}^{\infty}k^{p-2}|a_k|^p.
$$

As we can see, the conditions for for monotone functions and functions with monotone coefficients to belong to the space $L_p$  are the same, namely, the convergence of the series:
$$
\sum_{k=0}^{\infty}k^{p-2}|a_k|^p.
$$

For series by Haar sistem the situation is different. P.L. Ul'yanov in \cite{Uliyanov} proved, that if the Fourier-Haar coefficients $\left\{a_k\right\}_{k=1}^\infty$ are monotonous, then in order to the function $f\in L_p[0,1]$ at $1<p<\infty$ it is necessary and sufficient that $\left\{a_k\right\}_{k=1}^\infty\in l_2$, i.e. that the series $\sum_{k=1}^\infty|a_k|^2$ be converged.\\

Nursultanov E.D. and Aubakirov T.U. in paper \cite{Nursultanov1} proved the following statement:

{\it Let $1<p<\infty$, $f$ is a monotone function. Then in order to $f\in L_p[0,1]$ it is necessary and sufficient that for the sequence of its Fourier-Haar coefficients $\left\{a_k^j\right\}_{k=0, j=1}^{\infty, 2^k}$ the condition was met:
\begin{equation*}
\left(\sum_{k=0}^\infty\left(2^{k\left(\frac{1}{2}-\frac{1}{p}\right)}\sup_{1\leq j\leq 2^k}|a_k^j|\right)^{p}\right)^\frac{1}{p}<\infty.
\end{equation*}
}

For multiple trigonometric series, analogues of the Hardy-Littlewood theorem were obtained by F. Moritz \cite{Moricz}, M.I. Dyachenko \cite{Dyachenko1, Dyachenko2}. 

If the coefficients $a=\{a_{k_1,k_2}\}^{\infty,\infty}_{k_1=0 k_2=0}$ are monotonic by each variable index $k_1, k_2$, then as showed M.I.Dyachenko \cite{Dyachenko1}, the convergence of the number series 
$$
\sum^\infty_{k_1=1}\sum^\infty_{k_2=1}k_1^{p-2}k_2^{p-2}|a_{k_1 k_2}|^p
$$ 
is equivalent to $f\in L_p(\mathbb{T}^2)$ only when $\frac{4}{3}<p<\infty$ (in the case $2\leq n$, for $\frac{2n}{n+1}<p<\infty$).

The goal of our work is to obtain the Hardy-Littlewood theorem for multiple Fourier-Haar series in Lebesgue spaces $L_{\bar{p}}[0,1]^2$ with mixed metric and in anisotropic net spaces $N_{\bar{p},\bar{q}}(M)$.

\section{Main results}

The Haar system is a system of functions $\chi=\{\chi_n(x)\}^\infty_{n=1}$, $x\in[0,1]$, in which $\chi_1(x)\equiv 1$, and the function $\chi_n(x)$ where $n=2^k+j$, where $k=0,1,\ldots$, $j=1,2,\ldots,2^k$ is defined as:
$$
\chi_n(x)=\chi^j_k(x)=\left\lbrace
\begin{aligned}
    & 2^\frac{k}{2},\;\;  \frac{2j-2}{2^{k+1}}<x<\frac{2j-1}{2^{k+1}} \\
    & -2^\frac{k}{2}, \;\; \frac{2j-1}{2^{k+1}}<x<\frac{2j}{2^{k+1}} \\
    & 0, \;\; x\notin\left(\frac{j-1}{2^k};\frac{j}{2^k}\right)\\
\end{aligned}
\right.
$$

For the function  $f(x_1,x_2) \in L_1[0,1]^2$ consider its Fourier-Haar series
$$
f(x_1,x_2) \sim \sum_{k_1=0}^{\infty}\sum_{k_2=0}^{\infty}\sum_{j_1=1}^{2^{k_1}}\sum_{j_2=1}^{2^{k_2}}a_{k_1,k_2}^{j_1,j_2}(f)\chi_{k_1}^{j_1}(x_1)\chi_{k_2}^{j_2}(x_2),
$$
where $a_{k_1,k_2}^{j_1,j_2}(f)$ are the Fourier-Haar coefficients of the function $f(x_1,x_2)$ are defined as follows:
$$
a^{j_1,j_2}_{k_1,k_2}=\int_0^1\int_0^1 f(x_1,x_2)\chi_{k_1}^{j_1}(x_1)\chi_{k_2}^{j_2}(x_2)dx_1dx_2. 
$$

For $\bar{\sigma}=(\sigma_1,\sigma_2)\in\mathbb{R}^2$, $1\leq\bar{q}\leq\infty$ define space $l_{\bar{q}}^{\bar{\sigma}}(l_\infty)$, as a set of all sequences $a=\{a_{k_1k_2}^{j_1j_2}: k_i\in\mathbb{Z}_+,1\leq j\leq2^{k_i}, i=1,2\}$, for which is finite the norm 
$$
\|a\|_{l^{\bar{\sigma}}_{\bar{q}}(l_\infty)}=\left(\sum^\infty_{k_2=0}\left(\sum^\infty_{k_1=0}\left(2^{\sigma_1k_1+\sigma_2k_2}\sup_{{}^{1\leq j_i\leq\infty}_{i=1,2}}|a_{k_1k_2}^{j_1j_2}|\right)^{q_1}\right)^\frac{q_2}{q_1}\right)^\frac{1}{q_2},
$$

hereinafter expression $\left(\sum_{k=0}^\infty{b_k}^q\right)^\frac{1}{q}$ in the case, when $q=\infty$, is understood as $\sup_{k\geq 0}b_k$.

Let $M$ is set of all rectangles $Q=Q_1\times Q_2$ from $\mathbb{R}^2$, for function $f(x_1,x_2)$ integrable on each set $Q\in M$ define 
$$
\displaystyle\bar{f}(t_1,t_2; M)=\sup_{|Q_i|\geq t_i}\frac{1}{|Q_1||Q_2|}\left|\int_{Q_1}\int_{Q_2}{f(x_1,x_2)dx_1dx_2}\right|,\;\;\;\;t_i>0,
$$
where $|Q_i|$ is the $Q_i$ segment length.

Let $0<\bar{p}=(p_1,p_2)<\infty$, $0<\bar{q}=(q_1,q_2)\leq\infty$. By $N_{\bar{p},\bar{q}}(M)$ denote the set of all functions $f(x_1,x_2)$, for which
$$
\left\|f\right\|_{N_{\bar{p},\bar{q}}(M)}= \left(\int^{\infty}_{0}\left(\int^{\infty}_{0}\left(t_1^{\frac{1}{p_1}}t_2^{\frac{1}{p_2}}\bar{f}(t_1,t_2; M)\right)^{q_1}\frac{dt_1}{t_1}\right)^{\frac{q_2}{q_1}}\frac{dt_2}{t_2}\right)^\frac{1}{q_2}<\infty,
$$
hereinafter, when $q=\infty$, expression $\left(\int^\infty_0\left(\varphi(t)\right)^q\frac{dt}{t}\right)^\frac{1}{q}$ is understood as $\sup_{t>0}\varphi(t)$.

Spaces of this type were introduced and studied in \cite{Nursultanov2} and were called net spaces. Net spaces are an important research tool in the theory of Fourier series, in operator theory and in other directions.

\begin{theorem}\label{NPQ}
Let $1<\bar{p}<\infty$, $0<\bar{q}\leq\infty$, $\bar{\sigma}=\frac{1}{2}-\frac{1}{\bar{p}}$, $M$ is 
the set of all rectangles in $[0,1]^2$. Then, for $f\in N_{\bar{p},\bar{q}}(M)$ it is necessary and sufficient that the sequence of its Fourier-Haar coefficients $a=\{a_{k_1k_2}^{j_1j_2}: k_i\in\mathbb{Z}_+,1\leq j\leq2^{k_i}, i=1,2\}$ belonged to the space $l_{\bar{q}}^{\bar{\sigma}}(l_\infty)$, and is fulfilled the relation
\begin{equation}\label{T1}
\|f\|_{N_{\bar{p},\bar{q}}(M)}\asymp \|a\|_{l^{\bar{\sigma}}_{\bar{q}}(l_\infty)}.
\end{equation}
\end{theorem}

Note, that for spaces $N_{\bar{p}, \bar{q}}(M)$, the relation \eqref{T1} holds without any additional conditions on the function $f$ and its Fourier coefficients. Thus, for the net spaces $N_{\bar{p}, \bar{q}}(M)$, an analogue of Parseval's equality holds for all $1<\bar{p}<\infty$.

Let $0<\bar{p}\leq\infty$. The space $L_{\bar{p}}[0,1]^2$, called the Lebesgue space with a mixed metric, is defined as the set of functions $f$ measurable on $[0,1]^2$ for which is finite quantity
$$
\|f\|_{L_{\bar{p}}[0,1]^2}:=\left(\int_0^1\left(\int_0^1\left|f(x_1,x_2)\right|^{p_1}dx_1\right)^\frac{p_2}{p_1}dx_2\right)^\frac{1}{p_2}.
$$

The function $f(x_1,x_2)$ is called monotonically non-increasing by each variable, if  for $0\leq y_1\leq x_1$ and $0\leq y_2\leq x_2$ the inequality is holds
$$
f(x_1,x_2)\leq f(y_1,y_2).
$$

\begin{theorem}\label{LP}
Let $1<\bar{p}<\infty$, $\bar{\sigma}=\frac{1}{2}-\frac{1}{\bar{p}}$, $f(x_1,x_2)$ is monotonically non-increasing function by each variable. Then, for $f\in L_{\bar{p}}[0,1]^2$ it is necessary and sufficient that the sequence of its Fourier-Haar coefficients $a=\{a_{k_1k_2}^{j_1j_2}: k_i\in\mathbb{Z}_+,1\leq j\leq2^{k_i}, i=1,2\}$ belonged to the space $l_{\bar{p}}^{\bar{\sigma}}(l_\infty)$, and is fulfilled the relation
\begin{equation*}
\|f\|_{L_{\bar{p}}[0,1]^2}\asymp \|a\|_{l^{\bar{\sigma}}_{\bar{p}}(l_\infty)}.
\end{equation*}
\end{theorem}

\begin{remark}
The condition of a monotonically non-increasing function in \ref{LP} theorem can be replaced by monotonicity in each variable.
\end{remark}

\section{Interpolation theorems for anisotropic spaces}

In this section, we will consider an interpolation method for anisotropic spaces from the work \cite{Nursultanov4}. This method is based on ideas from the works of Sparr G. \cite{Sparr}, Fernandez D.L. \cite {Fernandez1, Fernandez2, Fernandez3} and others \cite {Cobus}.

Let ${\bf A_0} = (A_1^{0},A_2^{0}),\; {\bf A_1}=(A_1^1, A_2^1)$ are two anisotropic spaces, $E=
\{\varepsilon=(\varepsilon_1,\varepsilon_2):\varepsilon_i=0$, or $\varepsilon_i=1,\;\;i=1,2\}$. For arbitrary $\varepsilon\in E$ define the space ${\bf{A}_\varepsilon}=(A_1^{\varepsilon_1},A_2^{\varepsilon_2})$ with norm
$$
\|a\|_{\bf{A}_\varepsilon} = \|\|a\|_{A_1^{\varepsilon_1}}\|_{A_2^{\varepsilon_2}}.
$$

A pair of anisotropic spaces ${\bf{A_0}}=(A_1^0, A_2^0)$, ${\bf{A_1}}=(A_1^1, A_2^1)$
will be called compatible, if there is a linear Hausdorff space, containing as subsets the spaces $\bf{A_\varepsilon}$, $\varepsilon\in E$.

Let $0<\bar\theta=(\theta_1,\theta_2)< 1$, $0<\bar{q}=(q_1,q_2)\leq\infty$.
By ${\bf A}_{\bar{\theta},\bar{q}}=({\bf A_0}, {\bf A_1})_{\bar{\theta},\bar{q}}$ we 
denote the linear subset $\sum_{\varepsilon\in E}{\bf{A}}_\varepsilon$, for whose elements it is true:
$$
\|a\|_{A_{\bar{\theta},\bar{q}}} =\left(\int^{\infty}_{0}\left(\int^{\infty}_{0}\left(t_1^{-\theta_1}t_2^{-\theta_2}K(t_1,t_2)\right)^{q_1}\frac{dt_1}{t_1}\right)^\frac{q_2}{q_1}\frac{dt_2}{t_2}\right)^\frac{1}{q_2}<\infty,
$$
where
$$
K(t,a;{\bf A_0},{\bf A_1}) = \inf\{\sum_{\varepsilon\in E} t^{\varepsilon}\|a_{\varepsilon}\|_{{\bf{A}}_\varepsilon} \;:\; a=\sum_{\varepsilon\in E}a_\varepsilon,\;a_\varepsilon\in {\bf{A}}_\varepsilon\},
$$
where $t^{\varepsilon} = t_1^{\varepsilon_1} t_2^{\varepsilon_2}$.

We need a theorem from \cite{Bekmaganbetov}, which we formulate for our case, when $A =l_\infty$.

\begin{theorem}\label{BKA}
Let $\bar{\sigma}_0=(\sigma_1^0, \sigma_2^0)>\bar{\sigma}_1=(\sigma_1^1,\sigma_2^1)$, $0\leq\bar{q}, \bar{q}_0, \bar{q}_1\leq\infty$, $0<\bar{\theta}=(\theta_1,\theta_2)<1$ then will be true the equality
$$
\left(l_{\bar{q}_0}^{\bar{\sigma}_0}(l_\infty), l_{\bar{q}_1}^{\bar{\sigma}_1}(l_\infty)\right)_{\bar{\theta},\bar{q}}=l_{\bar{q}}^{\bar{\sigma}}(l_\infty),
$$
where $\bar{\sigma}=(1-\bar{\theta})\bar{\sigma}_0+\bar{\theta}\bar{\sigma}_1$.
\end{theorem}

The following interpolation theorem holds for anisotropic net spaces
\begin{theorem}\label{IT}
Let $M$ is the set of all rectangles in $\mathbb{R}^2$,  $1<\bar{p}_0<\bar{p}_1<\infty$, $1\leq\bar{q}_0,\bar{q},\bar{q}_1\leq\infty$, $0<\bar{\theta}=(\theta_1,\theta_2)<1$ then
\begin{equation*}
(N_{\bar{p}_0,\bar{q}_0}(M), N_{\bar{p}_1,\bar{q}_1}(M))_{\bar{\theta},\bar{q}}=N_{\bar{p},\bar{q}}(M),
\end{equation*}
where $\frac{1}{\bar{p}}=\frac{1-\bar{\theta}}{\bar{p}_0}+\frac{\bar{\theta}}{\bar{p}_1}$.
\end{theorem}

This theorem was proved in \cite{BashNurs}. The method of proving theorem \ref{IT} differs from the proof of theorem \ref{BKA}. In the proof of Theorem \ref{BKA}, the property of ideality of the spaces $l_{\bar{q}}^{\bar{\sigma}} (l_\infty)$ was essentially used, and the spaces $N_{\bar{p},\bar{q}}(M)$ are not ideal. Therefore, the proof uses other ideas.

\section{Proof of the theorem \ref{NPQ}}

Let function $f\in N_{\bar{p},\bar{\infty}}(M)$, $a(f)=\{a_{k_1k_2}^{j_1j_2}: k_i\in\mathbb{Z}_+,1\leq j\leq2^{k_i}, i=1,2\}$ are its Fourier coefficients by the Haar system.
Let us prove the inequality 
\begin{equation}
\left\|a(f)\right\|_{l^{\bar{\sigma}}_{\bar{\infty}}(l_\infty)}\leq C\left\|f\right\|_{N_{\bar{p},\bar{\infty}}(M)},
\end{equation}
where $\bar{\sigma}=\frac{1}{2}-\frac{1}{\bar{p}}$, $1<\bar{p}<\infty$.

By the definition of the space $l_{\bar{p}}^{\bar{\sigma}}(l_\infty)$:
$$
\|a(f)\|_{l^{\bar{\sigma}}_{\bar{\infty}}(l_\infty)}=\sup_{k_i\geq 0}2^{k_1\left(\frac{1}{2}-\frac{1}{p_1}\right)+k_2\left(\frac{1}{2}-\frac{1}{p_2}\right)}\max_{1\leq j_i\leq2^{k_i}}|a^{j_1,j_2}_{k_1,k_2}|.
$$

Note, that from the definition of the Fourier-Haar coefficients, we have
\begin{equation*}
\begin{split}
&a_{k_1,k_2}^{j_1,j_2}(f)=2^{\frac{k_1}{2}+\frac{k_2}{2}}\left[\int_{(\bigtriangleup_{k_1}^{j_1})^+}\int_{(\bigtriangleup_{k_2}^{j_2})^+}f(x_1,x_2)dx_1dx_2-\int_{(\bigtriangleup_{k_1}^{j_1})^-}\int_{(\bigtriangleup_{k_2}^{j_2})^+}f(x_1,x_2)dx_1dx_2-\right.\\
&\left.-\int_{(\bigtriangleup_{k_1}^{j_1})^+}\int_{(\bigtriangleup_{k_2}^{j_2})^-}f(x_1,x_2)dx_1dx_2+\int_{(\bigtriangleup_{k_1}^{j_1})^-}\int_{(\bigtriangleup_{k_2}^{j_2})^-}f(x_1,x_2)dx_1dx_2\right],
\end{split}
\end{equation*}
where 
$$
(\bigtriangleup_{k_1}^{j_1})^+= \left(\frac{2j_1-2}{2^{k_1+1}}; \frac{2j_1-1}{2^{k_1+1}}\right),\;\;
(\bigtriangleup_{k_1}^{j_1})^-=\left(\frac{2j_1-1}{2^{k_1+1}}; \frac{2j_1}{2^{k_1+1}}\right),
$$
$$
(\bigtriangleup_{k_2}^{j_2})^+= \left(\frac{2j_2-2}{2^{k_2+1}}; \frac{2j_2-1}{2^{k_2+1}}\right),\;\;
(\bigtriangleup_{k_2}^{j_2})^-=\left(\frac{2j_2-1}{2^{k_2+1}}; \frac{2j_2}{2^{k_2+1}}\right).
$$

Then, given that the lengths of the segments $|(\bigtriangleup)^+|=|(\bigtriangleup)^-|=\frac{1}{2^{k_i+1}}$, we get
\begin{equation*}
\begin{split}
&\|a(f)\|_{l^{\bar{\sigma}}_{\bar{\infty}}(l_\infty)}\leq 4\cdot2^{-\frac{1}{p'_1}-\frac{1}{p'_2}}\sup_{Q_1\times Q_2\in M}\frac{1}{|Q_1|^\frac{1}{p'_1}}\frac{1}{|Q_2|^\frac{1}{p'_2}}\left|\int_{Q_1}\int_{Q_2}f(x_1,x_2)dx_1dx_2\right|=\\
&=2^{\frac{1}{p_1}+\frac{1}{p_2}}\|f\|_{N_{\bar{p},\bar{\infty}}(M)}.
\end{split}
\end{equation*}

We define the operator $Tf=\{a_{k_1,k_2}^{j_1,j_2}(f)$. Let $\bar{p}$ satisfies the condition of the theorem and $\bar{p}_0$, $\bar{p}_1$, $\bar{\sigma}_0$, $\bar{\sigma}_1$ are such that $1<\bar{p}_0<\bar{p}<\bar{p}_1<\infty$, where $\bar{p}_0=(p_1^0,p_2^0)$, $\bar{p}_1=(p_1^1,p_2^1)$ and $\sigma_1^0=\frac{1}{2}-\frac{1}{p_1^0}, \sigma_2^0=\frac{1}{2}-\frac{1}{p_2^0},\sigma_1^1=\frac{1}{2}-\frac{1}{p_1^1},\sigma_2^1=\frac{1}{2}-\frac{1}{p_2^1}$.

Then it follows from the last inequality, that for a given operator:
$$
T: N_{(p_1^0, p_2^0),\infty}(M)\rightarrow l_{\bar{\infty}}^{(\sigma_1^0,\sigma_2^0)}(l_\infty),
$$
$$
T: N_{(p_1^0, p_2^1),\infty}(M)\rightarrow l_{\bar{\infty}}^{(\sigma_1^0, \sigma_2^1)}(l_\infty),
$$
$$
T: N_{(p_1^1, p_2^0),\infty}(M)\rightarrow l_{\bar{\infty}}^{(\sigma_1^1, \sigma_2^0)}(l_\infty),
$$
$$
T: N_{(p_1^1, p_2^1), \infty}(M)\rightarrow l_{\bar{\infty}}^{(\sigma_1^1,\sigma_2^1)}(l_\infty).
$$

Therefore, 
$$
T: \left(N_{\bar{p}_0, \bar{\infty}}(M), N_{\bar{p}_1, \bar{\infty}}(M)\right)_{\bar{\theta},\bar{q}}\rightarrow \left(l_{\bar{\infty}}^{\bar{\sigma}_0}(l_\infty), l_{\bar{\infty}}^{\bar{\sigma}_1}(l_\infty)\right)_{\bar{\theta},\bar{q}},
$$
 
According to the interpolation theorems \ref{BKA} and \ref{IT}, we have:
$$
\left(l_{\bar{\infty}}^{\bar{\sigma}_0}(l_\infty), l_{\bar{\infty}}^{\bar{\sigma}_1}(l_\infty)\right)_{\bar{\theta},\bar{q}}=l^{\bar{\sigma}}_{\bar{q}}(l_\infty), \;\;\; \left(N_{\bar{p}_0, \infty}(M), N_{\bar{p}_1, \infty}(M)\right)_{\bar{\theta},\bar{q}}=N_{\bar{p},\bar{q}}(M).
$$

Therefore, $T: N_{\bar{p},\bar{q}}(M)\rightarrow l^{\bar{\sigma}}_{\bar{q}}(l_\infty)$.

Thus, we get
$$
\|a(f)\|_{l^{\bar{\sigma}}_{\bar{q}}(l_\infty)} \leq C\|f\|_{N_{\bar{p},\bar{q}}(M)},
$$
where $\frac{1}{\bar{p}}=\frac{1-\bar{\theta}}{\bar{p}_0}+\frac{\bar{\theta}}{\bar{p}_1}$, $\bar{\sigma}=(1-\bar{\theta})\bar{\sigma}_0+\bar{\theta}\bar{\sigma}_1=\frac{1}{2}-\frac{1}{\bar{p}}$.

Let us show the reverse inequality. Let $N_1,N_2 \in\mathbb{N}$, $a=\{a_{k_1k_2}^{j_1j_2}: k_i\in\mathbb{Z}_+,1\leq j\leq2^{k_i}, i=1,2\} \in l^{\bar{\sigma}}_{\bar{q}}(l_\infty)$. Consider a polynomial
$$
S_{N_1,N_2}(a;x_1,x_2)=\sum_{k_1=0}^{N_1}\sum_{k_2=0}^{N_2}\sum_{j_1=1}^{2^{k_1}}\sum_{j_2=1}^{2^{k_2}}a_{k_1,k_2}^{j_1,j_2}\chi_{k_1}^{j_1}(x_1)\chi_{k_2}^{j_2}(x_2).
$$

$Q=Q_1\times Q_2$ an arbitrary rectangle from the net $M$. Then
\begin{equation*}
\begin{split}
&\frac{1}{|Q_1|^\frac{1}{p'_1}|Q_2|^\frac{1}{p'_2}}\left|\int_{Q_2}\int_{Q_1}S_{N_1,N_2}(a;x_1,x_2)dx_1dx_2\right|\\
&\leq\sum_{k_1=0}^{N_1}\sum_{k_2=0}^{N_2}\frac{1}{|Q_1|^\frac{1}{p'_1}|Q_2|^\frac{1}{p'_2}}\left|\int_{Q_2}\int_{Q_1}\sum_{j_1=1}^{2^{k_1}}\sum_{j_2=1}^{2^{k_2}}a^{j_1,j_2}_{k_1,k_2}\chi^{j_1}_{k_1}(x_1)\chi^{j_2}_{k_2}(x_2)dx_1dx_2\right|\\
&=\sum_{k_1=0}^{N_1}\sum_{k_2=0}^{N_2}\frac{1}{|Q_1|^\frac{1}{p'_1}|Q_2|^\frac{1}{p'_2}}\left|\sum_{j_1=1}^{2^{k_1}}\sum_{j_2=1}^{2^{k_2}}a^{j_1,j_2}_{k_1,k_2}\int_{Q_2}\chi^{j_2}_{k_2}(x_2)dx_2\int_{Q_1}\chi^{j_1}_{k_1}(x_1)dx_1\right|.
\end{split}
\end{equation*}

From the definitions of the functions $\chi^{j_1}_{k_1}(x_1)$ and $\chi^{j_2}_{k_2}(x_2)$ it follows that at most four terms in the sum 
$$
\sum_{j_1=1}^{2^{k_1}}\sum_{j_2=1}^{2^{k_2}}a^{j_1,j_2}_{k_1,k_2}\int_{Q_2}\chi^{j_2}_{k_2}(x_2)dx_2\int_{Q_1}\chi^{j_1}_{k_1}(x_1)dx_1
$$ 
are nonzero, namely, those terms where the supports of the functions $\chi^{j_1}_{k_1}(x_1)$, $\chi^{j_2}_{k_2}(x_2)$ contain, respectively, the ends of the segments $Q_1$, $Q_2$. Therefore,
\begin{equation*}
\begin{split}
&\frac{1}{|Q_1|^\frac{1}{p'_1}|Q_2|^\frac{1}{p'_2}}\left|\int_{Q_2}\int_{Q_1}S_{N_1,N_2}(a; x_1,x_2)dx_1dx_2\right|\leq\\
&\leq 4\sum_{k_1=0}^{N_1}\sum_{k_2=0}^{N_2}\frac{1}{|Q_1|^\frac{1}{p'_1}|Q_2|^\frac{1}{p'_2}}\max_{{}^{1\leq j_i\leq 2^{k_i}}_{i=1,2}}|a^{j_1,j_2}_{k_1,k_2}|2^{\frac{k_1}{2}+\frac{k_2}{2}}\min\left(|Q_1|,\frac{1}{2^{k_1}}\right)\min\left(|Q_2|,\frac{1}{2^{k_2}}\right).
\end{split}
\end{equation*}

Note, that 
$$
2^\frac{k_i}{2}\frac{1}{|Q_i|^\frac{1}{p'_i}}\min\left(|Q_i|,\frac{1}{2^{k_i}}\right)\leq 2^{k_i\left(\frac{1}{2}-\frac{1}{p_i}\right)},
$$
therefore,
$$
\frac{1}{|Q_1|^\frac{1}{p'_1}|Q_2|^\frac{1}{p'_2}}\left|\int_{Q_2}\int_{Q_1}S_{N_1,N_2}(a; x_1,x_2)dx_1dx_2\right|\leq 4\sum_{k_1=0}^{N_1}\sum_{k_2=0}^{N_2}2^{k_1\left(\frac{1}{2}-\frac{1}{p_1}\right)+k_2\left(\frac{1}{2}-\frac{1}{p_2}\right)}\max_{{}^{1\leq j_i\leq2^{k_i}}_{i=1,2}}|a^{j_1,j_2}_{k_1,k_2}|.
$$

Since the choice of the segments $Q_1$ and $Q_2$ is arbitrary, we obtain
$$
\|S_{N_1,N_2}(a;x_1,x_2)\|_{N_{\bar{p},\bar{\infty}}(M)}=\sup_{{}^{t_i>0}_{i=1,2}}{t_1}^\frac{1}{p_1}{t_2}^\frac{1}{p_2}\sup_{{}^{|Q_i|\geq t_i}_{i=1,2}}\frac{1}{|Q_1||Q_2|}\int_{Q_2}\int_{Q_1}S_{N_1,N_2}(a;x_1,x_2)dx_1dx_2
$$
$$
\leq\sup_{Q_1\times Q_2\in M}\frac{1}{|Q_1|^\frac{1}{p'_1}|Q_2|^\frac{1}{p'_2}}\left|\int_{Q_2}\int_{Q_1}S_{N_1,N_2}(a;x_1,x_2)dx_1dx_2\right|\leq 4\|a\|_{l^{\bar{\sigma}}_{\bar{1}}(l_\infty)}.
$$

Let $\bar{p}_0, \bar{p}_1, \bar{\sigma}_0, \bar{\sigma}_1$ such, that $1<\bar{p}_0<\bar{p}<\bar{p}_1<\infty$, $\bar{\sigma}_0=\frac{1}{2}-\frac{1}{\bar{p}_0}$, $\bar{\sigma}_1=\frac{1}{2}-\frac{1}{\bar{p}_1}$. Consider the operator $Ta=S_{N_1,N_2}(a; x_1,x_2)$. It follows from the last inequality that for this operator
$$
T: l_{\bar{1}}^{(\sigma_1^0,\sigma_2^0)}(l_\infty)\rightarrow N_{(p_1^0, p_2^0),\bar{\infty}}(M), 
$$
$$
T: l_{\bar{1}}^{(\sigma_1^0,\sigma_2^1)}(l_\infty)\rightarrow N_{(p_1^0,p_2^1), \bar{\infty}}(M), 
$$
$$
T: l_{\bar{1}}^{(\sigma_1^1,\sigma_2^0)}(l_\infty)\rightarrow N_{(p_1^1,p_2^0),\bar{\infty}}(M), 
$$
$$
T: l_{\bar{1}}^{(\sigma_1^1,\sigma_2^1)}(l_\infty)\rightarrow N_{(p_1^1, p_2^1),\bar{\infty}}(M).
$$

Then
$$
T: \left(l_{\bar{1}}^{\bar{\sigma}_0}(l_\infty), l_{\bar{1}}^{\bar{\sigma}_1}(l_\infty)\right)_{\bar{\theta},\bar{q}}\rightarrow \left(N_{\bar{p}_0, \bar{\infty}}(M), N_{\bar{p}_1, \bar{\infty}}(M)\right)_{\bar{\theta},\bar{q}},
$$

Hence, we have
$$
T: l_{\bar{q}}^{\bar{\sigma}}(l_\infty)\rightarrow N_{\bar{p},\bar{q}}(M),
$$
where $\frac{1}{\bar{p}}=\frac{1-\bar{\theta}}{\bar{p}_0}+\frac{\bar{\theta}}{\bar{p}_1}$, $\bar{\sigma}=(1-\bar{\theta})\bar{\sigma}_0+\bar{\theta}\bar{\sigma}_1=\frac{1}{2}-\frac{1}{\bar{p}}$ 
and therefore, the inequality holds
$$
\|S_{N_1,N_2}(a; x_1,x_2)\|_{N_{\bar{p},\bar{q}}(M)}\leq C\|a\|_{l^{\bar{\sigma}}_{\bar{q}}(l_\infty)}.
$$

Next, we use the fact that the space $N_{\bar{p},\bar{q}}(M)$ is a Banach space (see \cite{E.Nurs}) and therefore $S_{N_1,N_2}(a; x_1,x_2)$ converges to some function $f\in N_{\bar{p},\bar{q}}(M)$ for $N_1,N_2\rightarrow +\infty$.

\section{Proof of the theorem \ref{LP}}

First, we give a lemma.
\begin{lemma}\label{LT2}
Let $1<p<\infty$, $\varphi\in L_p[0,1]$, then
$$
\|\varphi\|_{L_p[0,1]}\asymp\left(\sum_{k=0}^\infty\left(2^{-\frac{k}{p}}\varphi^{**}(2^k)\right)^p\right)^\frac{1}{p},
$$
where 
$$
\varphi^{**}(t)=\frac{1}{t}\int_0^t\varphi^{*}(s)ds=\sup_{|e|=t}\frac{1}{|e|}\int_e|\varphi(x)|dx.
$$
\end{lemma}
\begin{proof}
$$
\|\varphi\|_{L_p[0,1]}\asymp\left(\int_0^1\left(\varphi^{**}(t)\right)^pdt\right)^\frac{1}{p}=\left(\sum_{k=0}^\infty\int_{2^{-(k+1)}}^{2^{-k}}\left(\varphi^{**}(t)\right)^pdt\right)^\frac{1}{p}\asymp\left(\sum_{k=0}^\infty 2^{-k}\left(\varphi^{**}(2^k)\right)^p\right)^\frac{1}{p}.
$$
\end{proof}
{\it Proof of the theorem \ref{LP}}

Note, that since $f(x_1,x_2)$ is monotonic, we have
$$
|a_{k_1k_2}^{j_1j_2}|=\left|\int_{\Delta_{k_2}^{j_2}}\int_{\Delta_{k_1}^{j_1}}f(x_1,x_2)\chi_{k_1k_2}^{j_1j_2}(x_1,x_2)dx_1dx_2\right|\leq
$$
$$
\leq 2^{\frac{k_1}{2}+\frac{k_2}{2}}\int_{\Delta_{k_2}^{j_2}}\int_{\Delta_{k_1}^{j_1}}|f(x_1,x_2)|dx_1dx_2\leq
$$
$$
\leq 2^{\frac{k_1}{2}+\frac{k_2}{2}}\int_0^{2^{-k_2}}\int_0^{2^{-k_1}}f(x_1,x_2)dx_1dx_2.
$$

Then from this estimate, the Minkowski inequalities and from the lemma \ref{LT2}, imply
\begin{equation*}
\begin{split}
&\|a(f)\|_{l_{\bar{p}}^{\bar{\sigma}}(l_\infty)}\leq \left(\sum_{k_2=0}^\infty\left(\sum_{k_1=0}^\infty\left(2^{\left(1-\frac{1}{p_1}\right)k_1+\left(1-\frac{1}{p_2}\right)k_2}\int_0^{2^{-k_2}}\int_0^{2^{-k_1}}f(x_1,x_2)dx_1dx_2\right)^{p_1}\right)^\frac{p_2}{p_1}\right)^\frac{1}{p_2}\leq\\
&\leq \left(\sum_{k_2=0}^\infty\left(\sum_{k_1=0}^\infty\left(2^{\left(1-\frac{1}{p_1}\right)k_1+\left(1-\frac{1}{p_2}\right)k_2}\int_0^{2^{-k_2}}\int_0^{2^{-k_1}}f(x_1,x_2)dx_1dx_2\right)^{p_1}\right)^\frac{p_2}{p_1}\right)^\frac{1}{p_2}\leq\\
&\leq\left(\sum_{k_2=0}^\infty\left(2^{\left(1-\frac{1}{p_2}\right)k_2}\int_0^{2^{-k_2}}\left(\sum_{k_1=0}^\infty\left(2^{\left(1-\frac{1}{p_1}\right)k_1}\int_0^{2^{-k_1}}f(x_1,x_2)dx_1\right)^{p_1}\right)^\frac{1}{p_1}dx_2\right)^{p_2}\right)^\frac{1}{p_2}=\\
&=\left(\sum_{k_2=0}^\infty\left(2^{1-\frac{k_2}{p_2}}\int_0^{2^{-k_2}}\varphi(x_2)dx_2\right)^{p_2}\right)^\frac{1}{p_2}\leq \left(\sum_{k_2=0}^\infty\left(2^{-\frac{k_2}{p_2}}\varphi^{**}(2^{-k})\right)^{p_2}\right)^\frac{1}{p_2}\asymp\\
&\asymp\left(\int_0^1\left(\varphi^*(t)\right)^{p_2}\right)^\frac{1}{p_2}=\|\varphi\|_{L_{p_2}},
\end{split}
\end{equation*}
here 
$$
\varphi(x_2)=\left(\sum_{k_1=0}^\infty\left(2^{1-\frac{k_1}{p_1}}\int_0^{2^{-k_1}}|f(x_1,x_2)|dx_1\right)^{p_1}\right)^\frac{1}{p_1}.
$$

Similarly,
$$
\varphi(x_2)\leq\left(\int_0^1|f(x_1,x_2)|^{p_1}dx_1\right)^\frac{1}{p_1}.
$$

Thus, we get
$$
\|a(f)\|_{l_{\bar{p}}^{\bar{\sigma}}(l_\infty)}\leq c\|f\|_{L_{\bar{p}}[0,1]^2}.
$$

Let us show the reverse inequality. Since $f(x_1,x_2)$ is a monotonically non-increasing by each variable function, then
$$
f(x_1,x_2)\leq\frac{1}{x_1x_2}\int_0^{x_1}\int_0^{x_2}f(y_1,y_2)dy_1dy_2\leq\bar{f}(x_1,x_2;M).
$$

Therefore, from theorem \ref{NPQ} it follows
$$
\|f\|_{L_{\bar{p}}[0,1]^2}\leq\|f\|_{N_{\bar{p},\bar{p}}(M)}\leq c\|a(f)\|_{l_{\bar{p}}^{\bar{\sigma}}(l_\infty)}.
$$

\renewcommand{\refname}{References}


\begin{thebibliography}{11}
\bibitem{Zigmund} A. Zygmund,    
\textit {Trigonometric series. Volume 2},
Cambridge University Press., 1959.

\bibitem{Uliyanov} P.L. Ulyanov,   
\textit {On series with respect to the Haar system},
Mat. Sb. [Math. USSR-Sb.], 63(105) (1964), no. 3, 356-391.

\bibitem{Nursultanov1} E. D. Nursultanov, T. U. Aubakirov, 
\textit {The Hardy–Littlewood theorem for Fourier–Haar series},
Math. Notes, 73:3 (2003), 314–320

\bibitem {Moricz} Moricz F. 
\textit {On double cosine, sine and Walsh series with monotone coefficients.}
Proc. Amer. Math. Soc, 109, no. 2 (1990),  417-425.

\bibitem {Dyachenko1} M.I. Dyachenko, 
\textit{On the convergence of double trigonometric series and Fourier series with monotone coefficients},
Math. USSR-Sb., 57:1 (1987), 57–75.

\bibitem {Dyachenko2} M.I. Dyachenko, 
\textit{Piecewise monotonic functions of several variables and a theorem of Hardy and Littlewood},
Math. USSR-Izv., 39:3 (1992), 1113–1128

\bibitem{Nursultanov2} E.D. Nursultanov,
{\it Net spaces and inequalities of Hardy–Littlewood type},
Sb. Math., 189:3 (1998), 399–419.

\bibitem{Nursultanov4} E.D. Nursultanov,
\textit {Interpolation theorems for anisotropic function spaces and their applications},
Reports of the Rus. Acad. of Sciences, 394:1 (2004), 1-4.

\bibitem {Sparr} Sparr G. 
\textit{Interpolation of several Banach spaces},
Ann. Mat. and Appl. V. 99 (1974), 247-316.

\bibitem{Fernandez1} D.L. Fernandez,  
\textit{Lorentz spaces with mixed norms}, 
J. Funct. Anal. 25:2 (1977), 128-146.

\bibitem{Fernandez2} D.L. Fernandez,  
\textit{Interpolation of $2^n$ Banach spaces}, 
Stud. Math. (PRL). 65:2  (1979), 175-201.

\bibitem{Fernandez3} D.L. Fernandez,  
\textit{Interpolation of $2^n$ Banach space and the Calderon spaces},
Proc.London Math. Soc. V. 56 (1988), 143-162.

\bibitem {Cobus} Cobus F., Peetre J. 
\textit{Interpolation compact operators: The multidimensional case}
 Proc. London Math. Soc. 63:2 (1991), 371-400.

\bibitem{Bekmaganbetov} K.A. Bekmaganbetov 
\textit{Interpolation theorem for  $l^\sigma_q(L_{p\tau})$ $L_{p\tau}(l^\sigma_q)$ spaces},
Bulletin of the Kazakh National University. Series mathematics, mechanics, computer science., 56:1 (2008), 30-42.

\bibitem{BashNurs}A.N. Bashirova, A.H. Kalidolday, E.D. Nursultanov
\textit{Interpolation theorem for anisotropic net spaces},  
{https://arxiv.org/abs/2009.00609}, (2020).

\bibitem{E.Nurs} M.Dyachenko, E.Nursultanov, S.Tikhonov,
\textit{Hardy-type theorems on Fourier transforms revised},
J. Math.Anal.Appl. 467 (2018), 171–184.

\end{thebibliography}
\end{document}